\documentclass{amsart}
\usepackage{latexsym,amssymb,amsthm,amsmath}

\usepackage{amssymb}
\usepackage{amsmath}

\theoremstyle{plain}
\newtheorem{theorem}{Theorem}[section]
\newtheorem*{Theorem B}{Theorem B}
\newtheorem*{Theorem A}{Theorem A}
\newtheorem{lemma}{Lemma}[section]

\newtheorem{corollary}{Corollary}[section]

\numberwithin{equation}{section}

\theoremstyle{remark}

\def\<{\langle }
\def\>{\rangle}
\def\({\left ( }
\def\){\right )}
\def\e{\eqref}

\title[Existence and uniqueness theorems]{Existence and uniqueness theorems for pointwise slant immersions in complex space forms}
\author[A. Alghanemi]{Azeb Alghanemi}  
\address{A. Alghanemi: Department of Mathematics, Faculty of Science, King Abdulaziz University, 21589 Jeddah, Saudi Arabia}
\email{aalghanemi@kau.edu.sa}
\author[N. M. Al-houiti]{Noura M. Al-houiti}
\address{N. M. Al-houiti: Department of Mathematics, Faculty of Science, King Abdulaziz University, 21589 Jeddah, Saudi Arabia}
\email{nalhouiti@ut.edu.sa}
\author[B.-Y. Chen]{Bang-Yen Chen}\address{B.-Y. Chen: Department of Mathematics, Michigan State University, 619 Red Cedar Road,   East Lansing, Michigan 48824--1027, U.S.A.}
\email{chenb@msu.edu}
\author[S. Uddin]{Siraj Uddin}
\address{S. Uddin: Department of Mathematics, Faculty of Science, King Abdulaziz University, 21589 Jeddah, Saudi Arabia}
\email{siraj.ch@gmail.com}

\begin{document}
\begin{abstract} An isometric immersion $f: M^{n} \rightarrow \tilde M^{m}$ from an $n$-dimensional Riemannian manifold $M^{n}$ into an almost Hermitian manifold $\tilde M^{m}$ of complex dimension $m$ is called {\it pointwise slant} if its Wirtinger angles define a function defined on $M$. In this paper we establish 
the existence and uniqueness theorems for pointwise slant immersions of Riemannian manifolds $M^{n}$ into a complex space form $\tilde M^{n}(c)$ of constant holomorphic sectional curvature $c$. 
\end{abstract}
\maketitle

\sloppy

\section{Introduction}
The class of slant submanifolds was initiated by B.-Y. Chen in \cite{C2} is an important class of submanifolds of almost Hermitian manifolds, which include almost complex and totally real submanifolds as special cases. In fact, Chen introduced a slant submanifold  $M$ of  an almost Hermitian manifold $(\tilde M, g, J)$ as a submanifold whose Wirtinger angle $\theta (X)$ between $JX $ and the tangent space $T_p M$ $(p\in M)$ is global constant, i.e, $\theta (X)$ is independent of the choice of  the nonzero vector $X \in T_p M$ and also independent of the choice of the point $p\in M$. Further,  Chen  and L. Vrancken established in  \cite{CV,CV2} the Existence and Uniqueness Theorem for slant immersions in complex space forms. Later,  in the contents of contact geometry, similar results were obtained for ``slant submanifolds'' in Sasakian space forms  \cite{Cab}, in Kenmotsu space forms \cite{PG},  as well as in cosymplectic space forms  \cite{GS}. 

Due to the popularity of slant submanifolds, F. Etayo \cite{E} defined the notion of pointwise slant submanifolds under the name of quasi-slant submanifolds as submanifolds whose Wirtinger angle $\theta (X)$  is independent of the choice  the nonzero vector $X \in T_p M $ at $p\in M$, but $\theta$ may depend on the point $p\in M$. Also, Etayo proved in \cite{E} that a complete, totally geodesic, quasi-slant submanifold of a Kaehler manifold is always a slant submanifold.

In \cite{CO}, B.-Y. Chen and O. J. Garay studied pointwise slant submanifolds of almost Hermitian manifolds and obtained many new results on such submanifolds. In particular, they provided many examples of pointwise slant submanifolds of almost Hermitian manifolds. Later on, pointwise slant submanifolds were investigated on Riemannian manifolds equipped with different structures  in \cite{PA1,S1}.

The main purpose of this paper is to establish the Existence and Uniqueness Theorems for pointwise slant immersions in complex space forms, which extend the Existence and Uniqueness Theorems of Chen and Vrancken. 

This paper is organized as follows: In Section 2, we recall some basic formulas and definitions for and pointwise slant submanifolds in almost Hermitian manifolds. In Section 3, we provide the  basic properties and formulas of pointwise slant submanifolds. In the last two sections, we prove
 the existence and uniqueness theorems for pointwise slant submanifolds in complex space forms, respectively.

\section{Preliminaries}

Let $\tilde M$ be an almost Hermitian manifold with an almost complex structure $J$ and an almost Hermitian metric $\<\,\; , \; \>$, which satisfy 
\begin{align}\label{3.1}
J^2  = -I,\;\;\;\;\<JX, JY\> =\<X, Y\>,
\end{align}
for any $X, Y$ be the vector fields on $\tilde M$.
An almost Hermitian manifold $\tilde M$ is called a {\it Kaehler} manifold if \cite{book11,Y} 
\begin{align}
\label{3.2}
(\tilde\nabla_X J)Y = 0,\;\; \forall X,Y\in TM,
\end{align}
where $\tilde\nabla$ denotes the Levi-Civita connection $\tilde M$. 

A Kaehler manifold is called a {\it complex space form} if it has constant holomorphic curvature. In the following, we shall denote  a complete simply connected $m$-dimensional complex space form with constant holomorphic curvature $c$ by $\tilde M^m(c)$.

The curvature tensor of $\tilde M^m(c)$ satisfies 
\begin{equation}\begin{aligned}\label{2.3}
\tilde R(X,Y)Z&=\frac{c}{4} \{\<Y,Z\>X-\<X,Z\>Y+\<JY,Z\>JX\\
&\hskip.3in -\<JX,Z\>JY+2\<X,JY\>JZ\}. 
\end{aligned}\end{equation}

Now, let $x:M\to \tilde M^{m}(c)$ be an isometric immersion of a Riemannian $n$-manifold  into a complex space form $\tilde M^m(c)$. We denote the {\it differential map} of $x$  by $x_{*}$ and let $T^\perp M$ denote  the normal bundle of $M$. 

For any $X\in TM$, we put
\begin{align}\label{2.4}
JX=PX+FX,
\end{align} 
where $PX$ and $FX$ denote the tangential and normal components of $JX$, respectively. 
Also, for any $V\in T^\perp M$, we write 
\begin{align}\label{2.5}
JV=tV+fV,
\end{align} 
where $tV$ and $fV$ are the tangential and normal components of $JV$, respectively. 
For $M$ in $\tilde M^m(c)$, let $\tilde\nabla$ and $\nabla$ be the Riemannian connections on $\tilde M$ and $M$,  respectively, while $\nabla^\perp$ is the normal connection in the normal bundle $T^\perp M$ of $M$. Then, the Gauss and Weingarten formulas are respectively given by 
\begin{align}\label{G}
&\tilde \nabla_X Y=\nabla_X Y+\sigma(X,Y),\;\;\;\\&\label{W} \tilde\nabla_XV=-A_VX+\nabla^\perp_XV,
\end{align} 
for any $X\in TM$ and $V\in T^\perp M$ such that $\sigma$ is the second fundamental form of $M$, and $A_V$ is the shape operator of the second fundamental form. 

It is well-known that the shape operator and the second fundamental form are related by 
\begin{align}\label{2.8}
\<A_VX, Y\>= \<\sigma(X,Y), V\>.
\end{align} 

For a submanifold $M$ in $\tilde M^m(c)$, let $R$ denote the curvature tensor of $M$, and $R^\perp$ denote the curvature tensor associated with the normal connection $\tilde \nabla$. Then the equation of Gauss, Codazzi and Ricci are  given respectively by \cite{C1*}
\begin{equation}\begin{aligned}\label{2.9}
\tilde R(X,Y; Z,W)=&\, R(X,Y;Z,W)+\<\sigma(X, Z), \sigma(Y, W)\>\\
&-\<\sigma(X, W), \sigma(Y, Z)\>,
\end{aligned}\end{equation}
\begin{align}\label{2.10}
(\tilde R(X,Y)Z)^\perp = (\bar\nabla_X\sigma)(Y, Z)-(\bar\nabla_Y\sigma)(X, Z)
\end{align}
and
\begin{align}\label{2.11}
\tilde R(X,Y;U,V)=R^\perp(X,Y;U,V)-\<[A_U,A_V]X,Y\>,
\end{align} 
for all $X, Y, Z, W \in TM$, and $U,V \in T^\perp M$, where ($\tilde R(X,Y)Z)^\perp $ is the normal component of $\tilde R(X,Y)Z$. 

The covariant derivative  $\bar\nabla \sigma$ of the second fundamental form $\sigma$ is defined by 
\begin{align}\label{2.12}
(\bar\nabla_X\sigma)(Y,Z)=\nabla_X^\perp\sigma(Y,Z)-\sigma(\nabla_XY,Z)-\sigma(Y,\nabla_XZ).
\end{align} 
The covariant derivatives of $P$ and $F$, respectively given by
\begin{align}\label{2.13}
&(\tilde\nabla_XP)Y=\nabla_XPY-P(\nabla_XY),
\\&\label{2.14}
(\tilde\nabla_XF)Y= \nabla^\perp_XFY-F(\nabla_XY).
\end{align} 
With the help of \eqref{3.2}--\eqref{W}, the above relations give (cf. \cite{C3})
\begin{align}\label{2.15}
&(\tilde\nabla_XP)Y= A_{FY}X + t\sigma(X,Y), 
\\& \label{2.16}
(\tilde\nabla_XF)Y= f\sigma(X,Y)- \sigma(X,PY),
\end{align} 
for  $X, Y \in TM$.

Let $M$ be an $n$-dimensional Riemannian manifold isometrically immersed into an almost Hermitian manifold $\tilde M$.
For a non-zero vector $X\in T_pM$, $p\in M$, the angle $\theta (X)$ between $JX$ and the tangent space $T_pM$ is called the {\it Wirtinger angle} of $X$. The Wirtinger angle is a real-valued function $\theta : T^*M \rightarrow R $, which is called the {\it Wirtinger function} defined on the set $T^*M$ consisting of all non-zero tangent vectors of $M$.

An isometric immersion $ f: M \rightarrow \tilde M$ is called {\it pointwise slant} if the Wirtinger angle  $\theta (X)$ can be regarded as a function on $ M$, which is known in \cite{CO} as the {\it slant function}. A pointwise slant submanifold with slant function $\theta$ is simply called a {\it pointwise $\theta$-slant submanifold.}

 Clearly, a pointwise slant submanifold $M$ is a slant submanifold if its slant function $\theta$ is a constant function on $M$ \cite{C2,C3}.  It is easy to verify that every surface of an almost Hermitian surface is pointwise slant (cf. Example 1 of \cite{C4}). 

A point $p$ of a submanifold $M$ of an almost Hermitian manifold is called a {\it totally real point} (resp.,  {\it complex point})  if  $\cos \theta=0$ (resp., $\sin \theta=0$) at $p$. 
A  submanifold $M$ of an almost Hermitian manifold is called a  {\it totally real submanifold}  if every point $p$ of $ M$ is totally real (cf. \cite{COg}). It is well-known that every pointwise slant submanifold is even-dimensional if it is not totally real (cf. Corollary 2.1 of \cite{C4}).

\section{Basics of pointwise slant submanifolds}

We recall the following basic result from  \cite{CO} for pointwise slant submanifolds of an almost Hermitian manifold.

\begin{lemma}\label{L3.1} Let $M $ be a submanifold of an almost Hermitian manifold. Then $M $ is pointwise slant if and only if 
\begin{align}\label{3.1}
P^2=- (\cos ^2\theta) I ,
\end{align} 
for some real-valued function $\theta$ defined on $M$, where $I$ is the identity map.
\end{lemma}
The following relations are direct consequences of equation \eqref{3.1}:
\begin{align}\label{3.2}
&\<PX,PY\>=(\cos^2\theta)\<X,Y\>,
\\&\label{3.3}
\<FX,FY\>=(\sin^2\theta_{}\<X,Y\>.
\end{align} 
Clearly, we also have
\begin{align}\label{3.4}
tFX=-(\sin^2{\theta})X,\;\;\;fFX=-FPX,
\end{align} 
for any vector field $X$ on $M$.

Now, we put 
\begin{align}\label{3.5}
X^*=({\csc \theta}){FX}.
\end{align} 
Let $\beta $ be the symmetric bilinear $TM$-valued form on $M$  defined by 
\begin{align}\label{3.6}
\beta(X,Y)= t\sigma(X,Y)
\end{align} 
for any $X,Y\in TM$. Then it follows from \eqref{2.4}, \eqref{3.5} and \e{3.6} that
\begin{align}\label{3.7}
J \beta (X,Y)= P\beta (X,Y)+(\sin\theta )\beta^*(X,Y).
\end{align} 
Also, if we let $\gamma$ be the symmetric bilinear $TM$-valued form on $M$ defined by
\begin{align}\label{3.8}
\gamma^* (X,Y)= f\sigma(X,Y),
\end{align} 
then  \eqref{2.5}, \eqref{3.6} and \e{3.8}, we find
\begin{align}\label{3.9}
J\sigma(X,Y)= \beta (X,Y)+\gamma^*(X,Y).
\end{align}
Applying the almost complex structure $J$ and using \eqref{3.1}, \eqref{2.5} and \eqref{3.7}, we get 
\begin{align*}
-\sigma(X,Y)=P\beta(X,Y)+(\sin\theta) \beta^*(X,Y)+t\gamma^*(X,Y)+f\gamma^*(X,Y).
\end{align*} 
Equating the tangential and the normal components, we obtain
\begin{align*}
P\beta(X,Y)=-t\gamma^*(X,Y),\;\;-\sigma(X,Y)=(\sin\theta)\beta^*(X,Y)+f\gamma^*(X,Y).
\end{align*}
Using \eqref{3.4} and \eqref{3.5}, we conclude that 
\begin{align*}
\gamma(X,Y)=(\csc\theta) P\beta(X,Y).
\end{align*} 
Also,
\begin{align}\label{3.10}
\sigma(X,Y)= - (\csc\theta)\beta^*(X,Y),
 \end{align} 
which can be written as 
\begin{align}\label{3.11}
\sigma(X,Y)=(\csc^2\theta)(P\beta(X,Y) - J\beta(X,Y) ).
\end{align} 
Taking the inner product of \eqref{2.15} with $Z\in TM$ and using \eqref{2.5}, \eqref{2.8}  and \eqref{3.6}, we derive that
\begin{align*}
\langle(\tilde\nabla_XP)Y,Z\rangle=\<\beta(X,Y), Z\>- \<\beta(X,Z),Y\>.
\end{align*} 

For an $n$-dimensional pointwise $\theta$-slant submanifold $M$ of a complex space form $\tilde M^m(c)$, we derive the equations of Gauss and Codazzi of $M$ in $\tilde M^m(c)$ as follows:

From \eqref{2.3}, we have 
\begin{align*}
\tilde R(X,Y;Z,W)&=\frac{c}{4} \big\{\<X,W\>\<Y,Z\>-\<X,Z\>\<Y,W\>+\<JX,W\>\<JY,Z\>\\ 
&\hskip.3in -\<JX,Z\>\<JY,W\>+2\<X,JY\>\<JZ,W\>\big\}.
\end{align*}
Substituting \eqref{2.9} into the above equation, we find
\begin{align*}
R(X,Y;Z&,W)+\<\sigma(X,Z),\sigma(Y, W)\>-\<\sigma(X,W),\sigma(Y,Z)\>\\
&=\frac{c}{4} \big\{\<X,W\>\<Y,Z\>-\<X,Z\>\<Y,W\>+\<PX,W\>\<PY,Z\>\\
&\hskip.3in -\<PX,Z\>\<PY,W\>+ 2\<X,PY\>\<PZ,W\>\big\}.
\end{align*}
Using \eqref{3.2}  and \eqref{3.11}, we may write 
\begin{align}\label{3.12}
R(X,Y;Z,W) &= (\csc^2\theta)\big\{\<\beta(X,W),\beta(Y,Z)\>-\<\beta(X,Z),\beta(Y,W)\>\big\}\notag\\
&+\frac{c}{4} \big\{\<X,W\>\<Y,Z\>-\<X,Z\>\<Y,W\>+\<PX,W\>\<PY,Z\>\\
&-\<PX,Z\>\<PY,W\>+2\<X,PY\>\<PZ,W\>\big\}, \notag
\end{align} 
which gives the Gauss equation of $M$ in $\tilde M^m(c)$. 

Next, for Codazzi equation if we take the normal parts of \eqref{2.3},  we obtain 
\begin{align}\label{3.13}
(\tilde R(X,Y;Z,W))^\perp&=\frac{c}{4}\big \{\<PY,Z\>FX-\<PX,Z\>FY+2\<X,PY\>FZ\big\}.
\end{align}
Further, it follows from \eqref{3.5} and \eqref{3.10} that 
\begin{align*}
\nabla^\perp_X(\sigma(Y, Z))= -\nabla^\perp_X( (\csc^2\theta) F\beta(Y,Z)),
\end{align*}
which yields
\begin{align*}
\nabla^\perp_X(\sigma(Y, Z))= - (\csc^2\theta) \nabla^\perp_XF\beta(Y,Z)+2 (\csc^2\theta \cot\theta) (X\theta)F\beta(Y,Z).
\end{align*}
Then by \eqref{2.16}, the above equation takes the form
\begin{align*}
\nabla^\perp_X(\sigma(Y, Z))&= - (\csc^2\theta) \Big[ f \sigma(X,\beta(Y,Z))- \sigma(X,P\beta(Y,Z))\\
&\hskip.5in  +F(\nabla_X\beta(Y,Z))-2(\cot\theta) (X\theta)F\beta(Y,Z)\Big]. 
\end{align*}

On the other hand, it also follows from \eqref{3.5} and \eqref{3.10} that
\begin{align*}
\sigma(\nabla_XY, Z)=-(\csc^2\theta) F\beta(\nabla_XY, Z).
\end{align*}
Similarly, we have
\begin{align*}
\sigma(Y, \nabla_XZ)= - (\csc^2\theta) F\beta(Y,\nabla_XZ).
\end{align*}
Substituting these relations into \eqref{2.12}, we obtain 
\begin{align*}
(\bar\nabla_X\sigma)(Y,Z)=&\, -(\csc^2\theta) \big[f\sigma(X,\beta(Y,Z))- \sigma(X,P\beta(Y,Z)) \\
& +F((\nabla_X\beta)(Y,Z))-2(\cot\theta) (X\theta)F\beta(Y,Z)]\big.
\end{align*}
Thus, by using \eqref{3.4}, \eqref{3.5} and \eqref{3.10}, we can write 
\begin{equation}\begin{aligned}\label{3.14}
(\bar\nabla_X\sigma)(Y,Z)=&\, -(\csc^2\theta) \big[(\csc^2\theta) FP\beta(X,\beta (Y,Z))\\
& + (\csc^2\theta) F\beta(X,P\beta (Y,Z))+F((\nabla_X\beta)(Y, Z))\\&-2(\cot\theta) (X\theta)F\beta(Y,Z)\big].
\end{aligned}\end{equation}
Similarly, we have
\begin{equation}\begin{aligned}\label{3.15}
(\bar\nabla_Y\sigma)(X,Z)=&-(\csc^2\theta)\big[(\csc^2\theta) FP\beta(Y,\beta (X,Z))\\&+ (\csc^2\theta) F \beta(Y,P\beta(X, Z))\\
&+F((\nabla_Y\beta)(X, Z))-2(\cot\theta) (Y\theta)F\beta(X,Z)\big].
\end{aligned}\end{equation}
Finally, after applying \eqref{3.13}, \eqref{3.14} and \eqref{2.11} into Codazzi's equation, we get 
\begin{align*}
&\hskip-.2in (\tilde\nabla_X\beta)(Y,Z)+(\csc^2\theta) \big\{P\beta(X,\beta(Y,Z))+\beta(X,P\beta (Y,Z))\big\}\\
&\hskip.2in +\frac{c}{4}(\sin^2\theta) \big\{\<X,PY\>Z+\<X,PZ\>Y\big\}-2(\cot\theta) (X\theta)\beta(Y,Z)\\
&= (\tilde\nabla_Y\beta)(X,Z)+(\csc^2\theta) \big\{P\beta(Y,\beta (X,Z))+\beta(Y,P\beta (X,Z))\big\}\\
&\hskip.2in +\frac{c}{4}(\sin^2\theta) \big\{\<Y,PX\>Z+\<Y,PZ\>X\big\}-2(\cot\theta) (Y\theta)\beta(X,Z).
\end{align*}


\section{Existence theorem}

In this section we present the detailed proofs of the existence and uniqueness theorems for pointwise slant immersions into a complex space form. 

\begin{theorem} {\rm{(Existence Theorem)}} \label{T:4.1}
Let $M^n$ be a simply connected Riemannian manifold of dimension $n$ equipped with metric tensor $\<\,\; ,\; \>$. Suppose that $c$ is a constant and there exist a smooth function $\theta$ on $M^{n}$ satisfying $0<\theta\leq \frac{\pi}{2}$,
an endomorphism $P$ of the tangent bundle $TM^{m}$ and a symmetric bilinear $TM^{n}$-valued form $ \beta $ on $M^{n}$ such that the following conditions are satisfied :
\begin{align}\label{4.1}
&P^2X=-(\cos^2\theta)X ,
\\&\label{4.2}
\<PX,Y\>= -\<X,PY\>,
\\&\label{4.3}
\langle (\tilde\nabla_XP)Y,Z\rangle=\<\beta(X,Y),Z\> -\<\beta(X,Z),Y\>,
\end{align} 
\begin{align}\label{4.4}
R(X,Y;Z,W) &= (\csc^2\theta)\big\{\<\beta(X,W),\beta(Y,Z)\>-\<\beta(X,Z),\beta(Y,W)\>\big\}\notag\\
&+\frac{c}{4} \big\{\<X,W\>\<Y,Z\>-\<X,Z\>\<Y,W\>+\<PX,W\>\<PY,Z\>\\
&- \<PX,Z\>\<PY,W\>+2\<X,PY\>\<PZ,W\>\big\}, \notag
\end{align} 
\begin{equation}\begin{aligned}\label{4.5}
&\hskip-.2in (\tilde\nabla_X\beta)(Y,Z)+(\csc^2\theta) \big\{P\beta(X,\beta (Y,Z))+\beta(X,P\beta (Y,Z))\big\}\\
&\hskip.2in  +\frac{c}{4}(\sin^2\theta) \big\{\<X,PY\>Z+\<X,PZ\>Y\big\}-2(\cot\theta) (X\theta)\beta(Y,Z)\\
&= (\tilde\nabla_Y\beta)(X,Z)+(\csc^2\theta) \big\{P\beta(Y,\beta (X,Z))+\beta(Y,P\beta (X,Z))\big\}\\&\hskip.2in +\frac{c}{4}(\sin^2\theta) \big\{\<Y,PX\>Z+\<Y,PZ\>X\big\}-2(\cot\theta) (Y\theta)\beta(X,Z),\end{aligned}\end{equation}
 for $X,Y,Z \in TM^{n}$. Then there exists a pointwise $\theta$-slant isometric immersion of $M^{n}$ into a complex space form $\tilde M^n(c)$ such that the second fundamental form $\sigma$ of $M^{n}$ is given by
 \begin{align}\label{4.6}
\sigma(X,Y)=(\csc^2\theta)(P\beta(X,Y) -J\beta(X,Y) ).
\end{align}
\end {theorem}
\begin{proof}
Assume that $c,\;\theta,\;P$ and $M^{n}$ satisfy the conditions  given in the theorem. Suppose that $TM^{n} \oplus TM^{n}$ be a Whitney sum.  For each $X\in TM^{n}$, we simply denote $(X,0)$ by $X$, $(0,X)$ by $X^*$, and the product metric on $TM^{n} \oplus TM^{n}$  by $\<\,\;, \;\>$. We define the 
endomorphism $\hat {J }$ on $TM^{n} \oplus TM^{n}$ by 
 \begin{align}\label{4.7}
\hat {J}(X,0)= (PX , (\sin \theta) X),\;\;\; \hat {J}(0,X)= (-(\sin \theta) X , -PX),
\end{align}
for each $X\in TM$. Then by \eqref{3.1}, \eqref{2.4} and \eqref{3.5}, we find 
\begin{align*}
&\hat {J}^2 ((X,0))= \hat {J}(PX , (\sin \theta) X)=-(X,0).
\end{align*}
Similarly, we find $ \hat {J}^2 ((0,X))=-(0,X)$. Hence, ${J}^2 =-I$. Also, it is easy to check that $\<\hat {J}X, \hat {J}Y\>=\<X,Y\>$ and it can be obtained by \eqref{4.7}. Therefore, $(\hat {J},\<\,\; , \;\>)$ is an almost Hermitian structure on $M$.

Now, we can define an endomorphism $A$ on $TM^{n}$, a $(TM^{n})^*$-valued symmetric bilinear form $\sigma$ on $TM^{n}$ and a metric connection $\nabla^\perp$ of the vector bundle 
$(TM^{n})^*$ over $M^{n}$ as follows :
 \begin{align}\label{4.8}
 &A_{Y^*}X= (\csc\theta) \{(\tilde\nabla_XP)Y - \beta(X,Y) \},
\\& \label{4.9}
\sigma(X,Y)=- (\csc\theta) \beta^*(X,Y) ,
\\& \label{4.10}
\nabla^\perp_XY^*=(\nabla_XY)^*-(\cot\theta) (X\theta)Y^*+(\csc^2\theta) \{P\beta^*(X,Y) +\beta^*(X,PY)\},
\end{align}
for $X,Y\in TM$. 

Denote by $\hat \nabla $ the canonical connection on $TM^{n} \oplus TM^{n}$ induced from equations \eqref{4.7}-\eqref{4.10}. Using \eqref{2.13}, \eqref{3.4}, \eqref{3.6} and \eqref{4.8}-\eqref{4.10}, we get 
\begin{align*}
(\hat \nabla _X \hat J)Y= (\hat \nabla _X \hat J)Y^*=0,
\end{align*}
for any $X,Y\in TM^{n}$.

Let $R^\perp$ be the curvature tensor corresponding to the connection  $\nabla^\perp$ on $(TM^{n})^*$, which gives by
\begin{align*}
R^\perp(X,Y)Z^*=\nabla^\perp_X\nabla^\perp_Y Z^*- \nabla^\perp_Y\nabla^\perp_X Z^*-\nabla^\perp_{[X,Y]}Z^*,
\end{align*}
for any $X,Y,Z\in TM^{n}$.
Then by \eqref{4.10}, we have 
\begin{align*}
R^\perp(X,Y)Z^*=\,&\nabla^\perp_X\big[(\nabla_Y Z)^*-\cot\theta (Y\theta)Z^*+\csc^2\theta \{P\beta^*(Y,Z)+\beta^* (Y,PZ)\}\big]\\
&-\nabla^\perp_Y\big[(\nabla_X Z)^*-\cot\theta (X\theta)Z^*+\csc^2\theta \{P\beta^*(X,Z)+\beta^* (X,PZ)\}\big]\\
& - (\nabla_{[X,Y]}Z)^*+\cot\theta ([X,Y]\theta)Z^*\\&- \csc^2\theta \{P\beta^*([X,Y],Z)+\beta^* ([X,Y],PZ)\}.
\end{align*}
Now, by \eqref{4.2}, \eqref{2.13}, \eqref{4.5} and \eqref{4.10} with direct calculations, we have the following relation
\begin{equation}\begin{aligned}\label{4.11}
 R^\perp(X,Y)Z^*= \,&(\csc^2\theta)\big[(Y\theta)-(X\theta)\big]Z^*+ (R(X,Y)Z)^*\\
&+\Big\{\frac{c}{4}P\big\{\<Y,PZ\>X-\<X,PZ\>Y-2\<X,PY\>Z\big\}\\
&+\frac{c}{4}\big\{\<Y,P^2Z\>X-\<X,P^2Z\>Y-2\<X,PY\>PZ\big\}\\
&+(\csc^2\theta) \big[(\tilde\nabla_XP)\beta(Y,Z)-(\tilde\nabla_YP)\beta(X,Z)\\
&-\beta(X,(\tilde\nabla_YP)Z)+\beta(Y,(\tilde\nabla_XP)Z)\big]\Big\}^*.
\end{aligned}\end{equation}

On the other hand, from \eqref{4.3} and \eqref{4.8}, we derive  
\begin{align}\label{4.12}
\langle [A_{Z^*}, A_{W^*}]X, Y\rangle = \,&(\csc^2\theta) \big\{\<(\tilde\nabla_XP)W,(\tilde\nabla_YP)Z\>-\<(\tilde\nabla_XP)Z,  (\tilde\nabla_YP)W\>\notag\\
 &+\<(\tilde\nabla_XP)Z,\beta(Y,W)\>+\<(\tilde\nabla_YP)W,\beta(X,Z)\>\notag\\
 &-\<(\tilde\nabla_XP)W,\beta(Y,Z)\>-\<(\nabla_YP)Z,\beta(X,W)\>\\
 &+\<\beta(X,W),\beta(Y,Z)\>-\<\beta(X,Z),\beta(Y,W)\>\big\}.\notag
\end{align}
Also, using \eqref{4.2}, we can write 
\begin{align*}
\<\beta(Y,Z),PW\>+\<P\beta(Y,Z),W\>=0.  
\end{align*}
Taking the covariant derivative of the above equation with respect to $X$ with using \eqref{2.13} and \eqref{4.2}, we obtain
\begin{align*}
\<\beta(Y,Z),(\tilde\nabla_XP )W\>+\<(\tilde\nabla_XP)\beta(Y,Z),W\>=0.
\end{align*}
Furthermore, from \eqref{4.3}, we find 
\begin{align*}
\<(\tilde\nabla_XP)Z,(\tilde\nabla_YP)W\>=\<(\tilde\nabla_XP)Z,\beta(Y,W)\>-\<\beta(Y,(\tilde\nabla_XP)Z),W\>.
\end{align*}
Substituting these relations in \eqref{4.11} and \eqref{4.12} with a simple computation, we arrive at 
\begin{align*}
&\hskip-.2in \<R^\perp(X,Y)Z^*, W^*\>-\<[A_{Z^*}, A_{W^*}]X, Y\>&\\
&=\frac{c}{4}\big[(\sin^2\theta) \{\<X,W\>\<Y,Z\>-\<X,Z\>\<Y,W\>\}-2\<X,PY\>\<PZ,W\>\big]\\
&+(\csc^2\theta)\big[Y\theta- X\theta\big]\<Z,W\>.
\end{align*}
Notice that the last equation with \eqref{2.3}, \eqref{4.1} and \eqref{4.2} means that $(M^{n}, A, \nabla^\perp )$ satisfies the Ricci equation of an $n$-dimensional pointwise $\theta$-slant submanifold of $\tilde M^n(c)$, while \eqref{4.4} and  \eqref{4.5} mean that $(M^{n}, \sigma)$ satisfies the equations of Gauss and Codazzi, respectively. Therefore, we have a vector bundle $TM^{n} \oplus TM^{n}$ over $M^{n}$ equipped with the product metric $\<\,\; ,\; \>$, the second fundamental form $\sigma$, the shape operator $A$, and the connections $\nabla^\perp$ and $\hat \nabla $ which satisfy the structure equations of $n$-dimensional pointwise $\theta$-slant submanifold of $\tilde M^n(c)$. Consequently, by applying Theorem 1 of \cite{ET} we conclude that there exists a pointwise $\theta$-slant isometric immersion from $M^{n}$ into $\tilde M^n(c)$ whose second fundamental form is given by $\sigma(X,Y)=(\csc^2\theta) (P\beta(X,Y) - J \beta(X,Y) )$.\end{proof}

A submanifold of an almost Hermitian manifold is called {\it purely real} if it contains no complex points (cf. \cite{C81}). It was proved in \cite{C4} that Ricci's equation is a consequence of the Gauss and Codazzi equations for purely real surfaces in any Kaehler surface. On the other hand, Theorem \ref{T:4.1} implies the following.

\begin{corollary}
The Ricci equation is a consequence of the Gauss and Codazzi equations for  n-dimensional pointwise slant submanifolds in any complex space form $\tilde M^{n}(c)$.
\end{corollary}

\section{Uniqueness theorem}

The next result provides the sufficient conditions to have the uniqueness property for pointwise slant  immersions.
\begin{theorem} {\rm (Uniqueness Theorem)}  \label{T:5.1} Let $\tilde M^n(c)$ be a complex space form and $M^n$ be a connected Riemannian n-manifold. Let $x^1,x^2: M^{n} \rightarrow \tilde M^n(c)$ be two pointwise $\theta$-slant isometric immersions with $0<\theta\leq \frac{\pi}{2}$. Suppose that $\sigma_1$ and $\sigma_2$ are the second fundamental forms of $x^1$ and $x^2$, respectively. If we have 
\begin{align}\label{5.1}
\<\sigma_1(X,Y),Jx^1_*Z\>=\<\sigma_2(X,Y),Jx^2_*Z\>,
\end{align}
for all $X,Y,Z \in TM$. In addition, if we consider that at least one of the following conditions is satisfied:
\begin{enumerate}
\item[(i)]  Every point is totally real point,
\item[(ii)] there exists a point $p$ of $M$ such that $P_1 = P_2 $,
\item[(iii)] $c\neq 0$,
\end{enumerate}
then $P_1 = P_2 $ and there exists an isometry $\phi$ of $\tilde M^n(c)$ such that $x^1=\phi(x^2)$.
\end {theorem}
\begin{proof} Let us choose any point $p\in M$ with assuming that $x^1(p)= x^2(p)$ and $x^1_*(p)= x^2_*(p)$. Then we take a geodesic $\psi$ through the point $p= \psi(0)$. Define $\psi_1=x^1(\psi)$ and $\psi_2=x^2(\psi)$. So, it is sufficient to prove that $\psi_1= \psi_2$ to prove the theorem. First, we know that $\psi_1(0)= \psi_2(0)$ and $ \psi'_1(0)= \psi'_2(0)$. We fix $\{e_1, e_2,...e_n\}$ be an orthonormal frame along $\psi$. Now, we define a frame along $\psi_1$ and $\psi_2$ as: 
$$a_i= x^1_*(e_i),\;\;~b_i= x^2_*(e_i), \;\;a_{n+i}=(x^1_*(e_i))^*,\;\;b_{n+i}=(x^2_*(e_i))^*,$$
such that $X^*$ defined by  \eqref{3.5}, for $i=1,2,\cdots n$.\\
From \eqref{4.9}, we can write 
\begin{align*}
\sigma_i= - (\csc\theta) (\beta_i)^*, 
\end{align*}
for any $i=1,2$. Using \eqref{4.9} and \eqref{4.1}, we obtain
\begin{align*}
\<(\beta_1)^*(X,Y),Fx^1_*Z\>=\<(\beta_2)^*(X,Y),Fx^2_*Z\>
\end{align*}
Then, by \eqref{3.5} the above equation takes the form 
\begin{align*}
\<\beta_1(X,Y),x^1_*Z\> = \<\beta_2(X,Y),x^2_*Z \>.
\end{align*}
As $Z$ be arbitrary vector field and $x^1_*(p)= x^2_*(p)$, we get $\beta_1=\beta_2$.

Now, we want to prove that $P_1 = P_2 $. If (i) is satisfied, then we have $P_1 = P_2 =0$, while if (ii) is satisfied, then from \eqref{4.3} we have $(\tilde\nabla_X(P_1-P_2))Y=0$. Since at any point $p\in M$ we already have $P_1 = P_2 $. So $P_1 = P_2 $ everywhere.

For the remaining situation, suppose that $c\neq 0$, $P_1\neq P_2 $ with (i) and (ii) are not satisfied. In the beginning we will show that $P_1 = -P_2 $. For this, using \eqref{4.4}, we find 
\begin{align*}
&\<P_1X,W\>\<P_1Y,Z\>-\<P_1X,Z\>\<P_1Y,W\>+2\<X,P_1Y\>\<P_1Z,W\>\\
&\hskip.2in = \<P_2X,W\>\<P_2Y,Z\>-\<P_2X,Z\>\<P_2Y,W\>+2\<X,P_2Y\>\<P_2Z,W\>.
\end{align*}
If we replace $W$ by $X$ and $Z$ by $Y$ with using \eqref{4.2}, the above equation can be written as 
\begin{align}\label{5.2}
\left(\<P_1X,Y\>\right)^2=\left(\<P_2X,Y\>\right)^2.
\end{align}
Then, we fix $e_1=X$, $e_2=P_1X$ and $e_3=Y$ with assuming that the component of $P_2e_1$ is in the same direction of a vector $e_3$ which is orthogonal to $e_1$ and $e_2$. Hence, equation \eqref{5.2} becomes
\begin{align*}
\left(\<P_2e_1,e_3\>\right)^2=\left(\<P_1e_1,e_3\>\right)^2= \left(\<e_2,e_3\>\right)^2=0,
\end{align*}
which is a contradiction. Thus, by \eqref{4.1} and \eqref{4.2}, we have $P_1u =\pm P_2 u$, for any $u \in T_pM$. Now, we let a basis $\{e_1, e_2,...,e_n \}$ of the tangent space at $p\in M$. Then there is a number $c_i\in \{-1,1\}$ such that $P_1e_i= c_i P_2e_i $. Therefore, we have 
\begin{align*}
 P_2 (e_i + e_j)= \pm P_1 (e_i + e_j) = c_i P_1e_i + c_jP_1e_j
\end{align*}
So, we conclude that all values of $c_i $ have to be equal. Hence, either $P_1u = P_2 u$ or $P_1u = -P_2 u$ for any $u \in T_pM$. As $M$ is connected, it follows that in situation (iii) either $P_1= P_2 $ or $P_1=-P_2$. 

If we assume that now we have two immersions such that $P_1 = - P_2 $. Then, we can write  \eqref{4.3} as 
\begin{align*}
\<(\tilde\nabla_XP_1)Y,Z\>=\<\beta_1(X,Y),Z\> -\<\beta_1(X,Z),Y\>.
\end{align*}
Similarly, we can obtain
\begin{align*}
\<(\tilde\nabla_XP_2)Y,Z\>=\<\beta_2(X,Y),Z\>-\<\beta_2(X,Z),Y\>.
\end{align*}
But $\beta_1= \beta_2=\beta$, therefore we deduce that 
\begin{align}\label{5.3}
\<\beta(X,Y), Z\>=\<\beta(X,Z),Y\>.
\end{align}
For both immersions, we rewrite the equation \eqref{4.5} as follows 
\begin{align*}
&\hskip-.3in \big\{(\tilde\nabla_X\beta_1)(Y,Z)- (\tilde\nabla_Y\beta_1)(X,Z)\big\}= (\csc^2\theta) \big\{P_1\beta_1(Y,\beta_1(X,Z))\\&+\beta_1(Y,P_1\beta_1(X,Z))-P_1\beta_1(X,\beta_1(Y,Z))-\beta_1(X,P_1\beta_1 (Y,Z))\big\} \\&+\frac{c}{4}(\sin^2\theta)\big\{\<Y,P_1Z\>X
-\<X,P_1Z\>Y-2\<X,P_1Y\>Z\big\}\\&+2(\cot\theta)\big\{ (X\theta)\beta_1(X,Z) - (Y\theta)\beta_1(Y,Z)\big\}.
\end{align*}
Similarly, we get
\begin{align*}
&\hskip-.3in \big\{(\tilde\nabla_X\beta_2)(Y,Z)- (\tilde\nabla_Y\beta_2)(X,Z)\big\}= (\csc^2\theta) \big\{P_2\beta_2(Y,\beta_2(X,Z))\\&+\beta_2(Y,P_2\beta_2(X,Z))-P_2\beta_2(X,\beta_2(Y,Z))-\beta_2(X,P_2\beta_2(Y,Z))\big\}\\& +\frac{c}{4}(\sin^2\theta)\big\{\<Y,P_2Z\>X
-\<X,P_2Z\>Y-2\<X,P_2Y\>Z\big\} \\&+2(\cot\theta)\big\{ (X\theta)\beta_2(X,Z)-(Y\theta)\beta_2(Y,Z)\big\}
\end{align*}
Putting $P_1=-P_2=P$ and $ \beta_1= \beta_2 = \beta$ in the above two equations, and then subtracting them, we obtain 
\begin{align*}
&P\beta(X,\beta(Y,Z))+ \beta(X,P\beta (Y,Z))-P\beta(Y,\beta (X,Z))- \beta(Y,P\beta (X,Z))\\
&\hskip.5in+\frac{c}{4}(\sin^4\theta)\big\{\<X,PZ\>Y- \<Y,PZ\>X+2  \<X,PY\>Z\big\}=0.
\end{align*}
Taking the inner product of the above equation with $W$ for any $W\in TM$ and applying \eqref{5.3}, we get
\begin{align}\notag
&\<\beta (X,Z),\beta(Y,PW)\>-\<\beta(X,PW),\beta(Y,Z)\>+ \<\beta(X,W),P\beta (Y,Z)\>\\
&\label{5.4}\hskip.2in  -\<\beta(Y,W),T\beta (X,Z)\>+\frac{c}{4}(\sin^4\theta)\big\{\<X,PZ\>\<Y,W\>- \<Y,PZ\>\<X,W\>\\ \notag
&\hspace{2cm}+2\<X,PY\>\<Z,W\>\big\}=0.
\end{align}
In the pervious equation, if $\beta=0$ at any point $p\in M$ , then we get a contradiction because $c\neq 0$.
So, we now choose a fixed point $p\in M$ and define a function $f$ on $UM_p$ by 
\begin{align*}
f(u)= \<\beta(u,u), u\>, 
\end{align*}
for each $u\in UM_p$, where $UM_p$ be the set of all unit tangent vectors. It is known that $UM_p$ is compact. So, there exists a vector $v $ such that $f $ arrives an absolute maximum at $v$. Suppose that $w$ be a unit vector orthogonal to $v$. Then the function $f(t) =f(g(t))$, such that $g(t) = (\cos t)v+ (\sin t)w$, satisfies the following
\begin{enumerate}
\item[(i)] $f^\prime(0) = 0$, which gives that $\<\beta(v, v),w\> = 0$.
\item[(ii)]$f^{\prime\prime}(0) \leq 0$, which implies that $\<\beta(v,w),w\> \leq \frac {1}{2}\<\beta(v, v), v\>$. 
\end{enumerate}
Now, by the total symmetry of $\beta$, we can fix an orthonormal basis $\{e_1= u, e_2,..., e_n\}$ which satisfies 
\begin{align}\label{5.5}
\beta(e_1,e_1) =\mu_1e_1,~~~~\beta(e_1,e_i) =\mu_ie_i , 
\end{align}
for $i >1$ and $\mu_i \leq \frac{1}{2}\mu_1$. Using the total symmetry of  \eqref{5.3} and $\beta \neq 0$, we obtain $\mu_1>0$. Substituting \eqref{5.3} and \eqref{5.5} into \eqref{5.4} with $X = Z = W = e_1$ and $Y = e_i$, we obtain 
\begin{align*}
\<\beta(e_i, Pe_1), \mu_1e_1\>-\<\beta(e_1,Pe_1), \mu_ie_i \>+\<\mu_1e_1,P\mu_ie_i\>- \<\mu_ie_i, P\mu_1e_1\>\\
+\frac{c}{4}(\sin^4\theta)\{-\<e_i,Pe_1\>\<e_1,e_1\>+2\<e_1,Pe_i\>\<e_1,e_1\>\}=0,
\end{align*}
which implies that 
\begin{align}\label{5.6}
\left(\mu_i^2 + \mu_1\mu_i+\frac{3c}{4}\sin^4\theta\right)\<e_i, Pe_1\>=0.
\end{align}
Now, we need to show that $Pe_1$ be an eigenvector of $\beta(e_1,\,\cdot \,)$. For this, we put $X = Z  = e_1$, $Y = e_i$ and $W= e_j$ in \eqref{5.4} such that $i , j > 1$ with using \eqref{5.3} and \eqref{5.5} to obtain 
\begin{align*}
 \mu_1\<\beta(e_1,e_i),Pe_j\>-\mu_i\<\beta(e_1,e_i),Pe_j\>+ \mu_i\<\beta(e_1,e_j),Pe_i\>- \mu_1\<\beta(e_i,e_j),Pe_1\>=0,
\end{align*}
or
 \begin{align}\label{5.7}
(\mu_i^2 - \mu_1\mu_i+ \mu_i\mu_j )\<e_i,Pe_j\>+ \mu_1\<\beta(e_i,e_j),Pe_1\>= 0.
\end{align}
Replacing the indices $i$ and $j$ in the above equation, we deduce that 
\begin{align}\label{5.8}
(\mu_j^2 - \mu_1\mu_j+ \mu_i\mu_j )\<e_i, Pe_j\>- \mu_1\<\beta(e_i,e_j),Pe_1\>= 0.
\end{align} 
Adding \eqref{5.7} and \eqref{5.8}, we find that
\begin{align*}
(\mu_i^2 - \mu_1\mu_i+\mu_j^2 - \mu_1\mu_j+ 2 \mu_i\mu_j )\<e_i, Pe_j\>= 0, 
\end{align*}
which gives that 
\begin{align}\label{5.9}
(\mu_i+ \mu_j) (\mu_1- \mu_i-\mu_j)\<e_i,Pe_j\>=0, 
\end{align} 
But we have $ \mu_i \leq \frac {1}{2}\mu_1$. So, $\mu_1- \mu_i-\mu_j=0 $ only if $\mu_i= \mu_j= \frac {1}{2}\mu_1$.

Now, by taking $X = W = e_1$, $Y = e_i$ and $Z= e_j$   such that $i, j > 1$ in \eqref{5.4}, we get 
\begin{equation}\begin{aligned}\label{5.10}
\<\beta(e_1,Pe_1),\beta(e_i,e_j) \>&- \mu_j\<\beta(e_i,e_j), Pe_1\> +\mu_i\mu_j\<e_i, Pe_j\>\\&\hskip-.3in 
+\mu_1\<\beta(e_i,e_j), Pe_1\>+\frac{c}{4}(\sin^4\theta)\<e_i, Pe_j\>= 0.
\end{aligned}\end{equation} 
Replacing the indices $i$ and $j$ in the above equation, we obtain
\begin{equation}\begin{aligned}\label{5.11}
\<\beta(e_1, Pe_1),\beta(e_i, e_j)\>& - \mu_i\<\beta(e_i,e_j), Pe_1\> +\mu_i\mu_j\<Pe_i,e_j\>\\&
\hskip-.3in +\mu_1\<\beta(e_i,e_j), Pe_1\>+\frac{c}{4}(\sin^4\theta)\<Pe_i,e_j\>= 0.
\end{aligned}\end{equation}
Subtracting \eqref{5.10} from \eqref{5.11}, we derive 
\begin{align}\label{5.12}
(\mu_i-\mu_j)\<\beta(e_i,e_j), Pe_1\>+2\mu_i\mu_j\<e_i, Pe_j\>+\frac{c}{2}(\sin^4\theta)\<e_i, Pe_j\>=0.
\end{align}
Now, we need to brief the preceding equations in the following method. First, interchanging $j$ by $i$ in \eqref{5.7} with using $\<Pe_i, e_i\> = 0$, we obtain 
\begin{align}\label{5.13}
\<\beta(e_i,e_i),Pe_1\>= 0 .
\end{align} 
Thus,  we have $\<\beta(\nu,\nu), Pe_1\>= 0$ if $ \nu$ is an eigenvector of $\beta(e_1, \cdot\, \,)$.  In addition, the symmetry of $\beta$ give us that $\<\beta(e_i,e_j),Pe_1\>=0$, whenever $\mu_i= \mu_j $.
So that, we can consider the following four different cases:
\begin{enumerate}
\item[(a)] $\mu_i+\mu_j \neq 0$, but not $ \mu_i = \mu_j =\frac {1}{2}\mu_1$. Thus, \eqref{4.9} follows $\<Pe_i,e_j\>= 0$;
\item[(b)]  $\mu_i+\mu_j =0 $ and $\mu_i \neq 0$. So, \eqref{5.7} gives that $\<\beta(e_i,e_j), Pe_1\> = \mu_i\<e_i, Pe_j\>$, if we apply this in \eqref{5.12}, we get $\<e_i, Pe_j\>= 0$;
\item[(c)] $\mu_i = \mu_j =0 $. Thus, by \eqref{5.12}, we obtain $\<Pe_i, e_j\>= 0$;
\item[(d)] $ \mu_i = \mu_j =\frac {1}{2}\mu_1$
\end{enumerate}
Hence, if we let $e_{i_1}, ..., _{i_r}$ are eigenvectors in an eigenvalue which is different from $\frac {1}{2}\mu_1$, then 
each $Pe_{i_s} , s=1,... , r$, can just have a component in the same direction of $e_1$, such as $Pe_{i_s}= \mu_{s}e_1$. Therefore, $\mu_{s}Pe_1= -( \cos ^2\theta) e_{i_s}$. 
 Accordingly, either $r = 1$ or there does not exit an eigenvector with eigenvalue different from $\frac {1}{2}\mu_1$. If $r = 1$, then certainly $Pe_1$  is an eigenvector.
In the other case $\beta(e_1, \cdot \,)$ limited to the space $e_1^\perp$, only has one eigenvalue, that $\frac {1}{2}\mu_1$.
As  $Pe_1$ is orthogonal to $e_1$ forever, then $Pe_1$ is also an eigenvector in this case. Thus, $Pe_1$ is always an eigenvector of  $\beta(e_1, \cdot \,)$.

Now, we can consider that $e_2$ is in the same direction of $Pe_1$. So, we get directly that $\beta(e_1, Pe_1)= \mu_2Pe_1$, such that from \eqref{4.6}, $\mu_2$ satisfies the equation
\begin{align}\label{5.14}
\mu_2^2 + \mu_1\mu_2+\frac{3c}{4}\sin^4\theta = 0. 
\end{align} 
If we put  $X = Z = e_1$, $Y = e_i$ and $W = Pe_1$, for $i > 2$ in \eqref{5.4},  we find 
\begin{align*}
\mu_i\<\beta(e_1, Pe_1),Pe_i\>- \mu_1\<\beta(e_i, Pe_1), Pe_1\>=0, 
\end{align*}
which gives 
\begin{align*}
\mu_i\mu_2(\cos ^2\theta)\<e_1,e_i\>- \mu_1\<\beta(e_i, Pe_1),Pe_1\>=0. 
\end{align*}
Hence, 
\begin{align*}
\beta(Pe_1,Pe_1)=\mu_2(\cos ^2\theta) e_1
\end{align*}
Taking $X = Z = W = Pe_1$, and $Y = e_1$ in \eqref{5.4} again, we derive
\begin{align}\label{5.15}
-\mu_2^2 - \mu_1\mu_2+\frac{3c}{4}\sin^4\theta = 0 .
\end{align} 
Finally, \eqref{5.14} and \eqref{5.15} implies that $c\sin^4\theta = 0$, which is a contradiction because $c \neq 0$ and $\theta$ be a real-valued function. Thus, $P_1=P_2$.
\end{proof}

\vskip.1in
\noindent {\bf Acknowledgement.} \rm{This project was funded by the Deanship of Scientific Research (DSR) at King Abdulaziz University, Jeddah, under grant no. (KEP-PhD-41-130-38). The authors, therefore, acknowledge with thanks to DSR for technical and financial support.}

\end{document}